\documentclass{amsart}
\usepackage{graphicx} 

\usepackage{stackengine}
    
\usepackage{amsmath,amssymb}
\usepackage{array} 
\usepackage{ amssymb }
\usepackage{amsfonts}
\usepackage{amsthm}
\usepackage{mathabx}
\usepackage{soul}
\usepackage{graphics}
\usepackage{graphicx}
\usepackage{float}
\usepackage[shortlabels]{enumitem}

\usepackage{marginnote}

\usepackage{enumitem}

\usepackage{ mathrsfs }

\usepackage[
colorlinks=true,
linkcolor=blue,
 citecolor=blue,
  urlcolor=blue,
     pagebackref,
]{hyperref}

\usepackage[normalem]{ulem}
\usepackage{bm}
  \newcolumntype{P}[1]{>{\centering\arraybackslash}p{#1}}
    \newcolumntype{M}[1]{>{\centering\arraybackslash}m{#1}}

\title[Restriction Theorem at zero Fourier dimension]{Slowly decaying Rajchman measures and a restriction theorem for the Fourier transform at the limit case of zero Fourier dimension}

\author{Iván Polasek}
\address[Iván Polasek]{Department of Mathematics,
University of Buenos Aires and IMAS-CONICET-UBA}
\email{ipolasek@dm.uba.ar}
\author{Ezequiel Rela}
\address[Ezequiel Rela]{Department of Mathematics,
University of Buenos Aires and IMAS-CONICET-UBA}
\email{erela@dm.uba.ar}
\address[Ezequiel Rela]{Guangdong Technion Israel INstitute of Technology, Shantou, China. } 
\email{ezequiel.rela@gtiit.edu.cn}


\newtheorem{theorem}{Theorem}[section]
\newtheorem{lemma}[theorem]{Lemma}

\theoremstyle{definition}

\newtheorem*{remark}{Remark}

\newenvironment{proof 1}{\paragraph{Proof:}}{\hfill$\square$}

\graphicspath{{Figuras/}}

\begin{document}
\begin{abstract}
    In this article we prove the existence of sets $E \subseteq \mathbb{R}$ of zero Fourier dimension such that it is possible to restrict the Fourier transform to $E$ on a certain non-trivial range $[1,\tilde{p})$ with $1<\tilde{p}<2$. This builds upon Mockenhaupt's Restriction Theorem; while this theorem could only be applied to sets of positive Fourier dimension, we show that the existence of a measure with polylogarithmic Fourier decay combined with full Hausdorff dimension 1 on the real line is enough to guarantee restriction. In order to achieve this, we combine two different tools: a modification of a construction from a recent work of Li and Liu to produce a set with specific Hausdorff and Fourier dimensions, and a generalization of the Stein-Tomas-Mockenhaupt Restriction Theorem.
\end{abstract}
\maketitle

\section{Introduction and Main Results}

The Restriction Problem is a long-standing one that encompasses many distinct but related questions. It concerns the question of whether, for certain $E \subseteq \mathbb{R}^n$, $1 \le p < 2$, it is possible to restrict to $E$ the Fourier transform $\widehat{f}$ of every $f \in L^p (\mathbb{R}^n)$. One of the first substantial answers to this question can be traced back to the work of Stein and Tomas \cite{tom75} and asserts that, for every $\frac{1}{p} \ge \frac{n+3}{2n+2}, \frac{1}{q} \ge \frac{n+1}{n-1}\left( 1-\frac{1}{p} \right)$ we have that
\begin{equation}\label{eq: REpq notation sphere}
    \| \widehat{f} \|_{L^q (\mathbb{S}^{n-1}, d\sigma)} \lesssim \| f \|_{L^p (\mathbb{R}^n)} \quad \forall  f \in L^p (\mathbb{R}^n),
\end{equation} 
where $\mathbb{S}^{n-1}$ is the unit sphere of $\mathbb{R}^n$ and $d\sigma$ is its uniform measure. 

This initial result has been studied in more detail and generalized in many different directions. The range of $p$ and $q$ proved by Stein and Tomas is conjectured to actually be larger, namely $\frac{1}{p} \ge \frac{n+1}{2n}, \frac{1}{q} \ge \frac{n+1}{n-1}\left( 1-\frac{1}{p} \right)$. This is the so called Restriction Conjecture for the sphere; it has been proven for $n=2$ (see \cite{fef70}) and remains open for higher dimensions. In terms of the possible subsets $E \subseteq \mathbb{R}^n$ for which a restriction theorem might hold, it is known (see \cite{Stein-PrincetonBook}) that \eqref{eq: REpq notation sphere} holds for the same range of $p$ and $q$ when $\mathbb{S}$ is replaced by a compact subset of any hypersurface whose Gaussian curvature is nowhere zero. A trivial computation shows that curvature is somehow crucial here, since no restriction theorem holds for a flat surface. In general, when $E$ is a compact subset of $\mathbb{R}^n$ and there exists a measure $\mu$ supported on $E$ such that 
\begin{equation}\label{eq: REpq notation general}
    \| \widehat{f} \|_{L^q (E, d\mu)} \lesssim \| f \|_{L^p (\mathbb{R}^n)} \quad \forall  f \in L^p (\mathbb{R}^n)
\end{equation} 
holds for a certain pair of values $(p,q)$, we will say that a restriction theorem $R_E (p \rightarrow q)$ holds. The problem has also been studied in the context of particular hypersurfaces and curves or with different hypotheses. See for example \cite{dru85}, \cite{Tab85} and \cite{Gre81}.

All of these results relied on the geometry of the sets $E$ as subsets of $\mathbb{R}^n$, $n\ge 2$. In contrast with the remark on curvature above, Mockenhaupt's breakthrough paper \cite{moc00} showed that it is possible to prove restriction theorems for fractal subsets of the real line, for which the notion of curvature makes no sense. Mockenhaupt's Restriction Theorem relies on the existence of a measure supported on the set which satisfies two hypotheses related to the Hausdorff and Fourier dimension of the set. This approach allows to extend the analysis to even more general scenarios without any geometric structure, including the case of the abstract Fourier transform defined on groups (see \cite{papa2010}). It is common to refer to this theorem as the Stein-Tomas-Mockenhaupt theorem (STM), we include it here for reference.

\begin{theorem}\label{thm:STM}
    Let $\mu$ be a compactly supported measure on $E \subseteq \mathbb{R}^n$ such that
    \begin{enumerate}[(i)]
        \item $\mu (B(x,R)) \lesssim R^\alpha$ for some $\alpha>0, \forall x \in \mathbb{R}^n, R>0$.
        \item $|\widehat{\mu}(\xi)| \lesssim |\xi|^{-\beta/2}$ for some $\beta >0, \forall \xi \neq 0.$
    \end{enumerate}
    Then, a restriction theorem $R_E (p \rightarrow 2)$ holds for every $1 \le p < \frac{2(2n-2\alpha + \beta)}{4(n-\alpha)+\beta}$.
\end{theorem}
We include here a remark on the hypotheses in STM theorem:  hypothesis $(i)$ is called the Frostman condition, and it is well known that it implies $\dim_H (E) \ge \alpha$. Hypothesis $(ii)$ is a result on the Fourier decay of the measure $\mu$, and as such guarantees that $\dim_F (E) \ge \beta$ (for the precise definitions of the Hausdorff dimension $\dim_H$ and  the Fourier dimension $\dim_F$, see Section \ref{sec: def and res}). A set with coinciding Hausdorff and Fourier dimension is called a \emph{Salem} set.

There is an important reason for the theorem above to be of the form $R_E (p \rightarrow 2)$ and that is the $T^*T$ method. This is a well know fact that says that an operator $T$ is bounded from $L^p$ to $L^2$ if and only if the composition $T^*T$ is bounded from $L^p$ to $L^{p'}$. In the context of the restriction operator, it is not difficult to check that the problem reduces to obtain $(p,p')$ bounds for the convolution operator $f\to f*\widehat{d\mu}$. The interested reader could check the details in the Stein-Tomas-Muckenhaupt theorem to verify that this $(p,p')$ bound is a consequence of the interpolation between $(1,\infty)$ and $(2,2)$ norm inequalites for this convolution operator. The decay of $\widehat{d\mu}$ becomes evidently useful to control the supremum norm, whilst the local growth condition of $\mu$ appears when controlling the $L^2$ norm of the convolution as a consequence of certain control of the $L^\infty$ norm of the Fourier transform of  kernel.

Theorem \ref{thm:STM} has also been improved. Bak and Seeger \cite{BS-endpoint} have proved that the theorem holds as well for the endpoint $p_* = \frac{2(2n-2\alpha + \beta)}{4(n-\alpha)+\beta}, \alpha < n$. Hambrook and Łaba \cite{HamLab2013} worked on the optimality of the range of $p$ for the family of Salem sets in the real line. They were able to prove the existence of a Salem set $E \subseteq \mathbb{R}$ of dimension $\alpha \in (0,1)$ supporting a measure $\mu$ such that $(i)$ holds for the dimension $\alpha$, $(ii)$ holds for every $\beta < \alpha$ but the restriction theorem $R_E(p \rightarrow q)$ fails for every $p>\frac{2(2n-2\alpha + \beta)}{4(n-\alpha)+\beta}$. Chen \cite{che16} later improved on this result, proving the sharpness of the restriction theorem in the same sense for general sets, not necessarily Salem sets. Both of these proofs relied on probabilistic arguments. Fraser, Hambrook and Ryou \cite{FHR-2024} went a step futher and were able to construct deterministic subsets of the real line which explicitly show the sharpness of the Restriction theorem. 

In particular, according to hypotheses $(i)$ and $(ii)$, in order to apply STM theorem one should be working with a set $E$ where both dimensions are strictly positive. Our main purpose here is to focus on the fact that decay of the Fourier transform is a central aspect of a restriction theorem. Hence, we want to push STM theorem to the limit case of zero Fourier dimension, but with a mild decay so we still have a chance to get restriction.
A first step in this direction was achieved by the second author, whose doctoral thesis contains the following generalized version of STM Restriction theorem (\cite{Rela2010}). 

\begin{theorem}\label{thm: generalized mockenhaupt}
    Let $E \subseteq \mathbb{R}^n$ and $\mu$ a measure supported on $E$. Let $h$ be a doubling dimension function (see Section \ref{sec: def and res} for the definition), and $g: \mathbb{R}_{\ge 0} \longrightarrow \mathbb{R}_{>0}$ a decreasing function such that 
    \begin{itemize}
        \item $\mu(B(x,R)) \lesssim h(R) \ \forall x \in \mathbb{R}^n, R>0$
        \item $|\widehat{\mu}(\xi)| \lesssim g(|\xi|) \ \forall \xi \in \mathbb{R}^n$. 
    \end{itemize} 
    Let $\Gamma$ be the sequence defined by $$\Gamma_k=g(2^{k-1})^{\frac{2}{p}-1}(2^{nk}h(2^{-k}))^{2-\frac{2}{p}}.$$ Then if $\Gamma\in\ell^1$ for some $p$, there is a restriction theorem $R_E(p\to 2)$.
\end{theorem}

For a long time, it was unclear how to apply this result in order to obtain a restriction theorem for a zero-dimensional set. The main purpose of this paper is to construct a family of zero-Fourier dimensional sets $E$ for which Theorem \ref{thm: generalized mockenhaupt} holds, provided that the measure verifies a sufficiently fast polylogarithmic Fourier decay and a Frostman condition that will force $\dim_H (E) = 1$. This condition on the Hausdorff dimension seems unavoidable, at least with the approach presented here. The reason can be found by noticing that the sequence $\Gamma_k$ above includes a factor of the form $2^{nk}$. A careful exploration of the proof shows that this factor comes from the effect of a dilation under the Fourier transform, making quantitatively explicit the dependance on the dimension of the ambient space. In the case of the real line, that dimension is 1, hence the condition on the dimension function $h$. In the same spirit as in STM result, the relevance of Theorem \ref{thm: generalized mockenhaupt} relies on the existence of a set satisfying the hypothesis. That is the main contribution of this article that we present here. 
\begin{theorem}\label{thm: main theorem}
   Let $r>1, a>0$. There exist a compact set $E\subset \mathbb{R}$ with $\dim_H(E)=1$ and $\dim_F(E)=0$ such that there is a restriction theorem $R_E(p \rightarrow 2)$ for $1 \le p < 1 + \frac{r-1}{1+r+2ar}$.
\end{theorem}

This restriction theorem will be a direct consequence of  Theorem \ref{thm: generalized mockenhaupt} and the specific construction we provide in Theorem \ref{thm: our construction theorem}. We postpone the statement of the theorem about the construction of the set to Section \ref{sec: def and res}, after we introduce some necessary definitions. This construction is based on a recent work of Li and Liu \cite{LL25}; in their paper, they construct subsets of the real line with predetermined Hausdorff and Fourier dimension, supporting a measure that captures both dimensions simultaneously. Our construction involves a measure whose behaviour can be controled up to a logarithmic scale, detecting the precise level of nuance that is needed for our purposes.

\section{Definitions and the specific construction}\label{sec: def and res}

We introduce here some definitions needed to present the main construction.

Let $E \subseteq \mathbb{R}^n$, $s \in [0,n]$. The $s$-dimensional Hausdorff measure is defined as follows:
    \begin{equation*}
        \mathcal{H}^s_\delta (E) = \inf \left\{ \sum_i (\textnormal{diám}(E_i))^s, E \subseteq \bigcup_i E_i, \textnormal{diám}(E_i) < \delta \right\},
    \end{equation*}

    \begin{equation*}
        \mathcal{H}^s (E) = \lim_{\delta \rightarrow 0^+} \mathcal{H}^s_\delta (E).
    \end{equation*}
It is an important result that 
    \begin{equation*}
        \mathcal{H}^s(E)>0 \iff \exists \ \mu \in \mathcal{M}(E)/ \mu(B(x,R)) \lesssim R^s \ \forall \ x\in \mathbb{R}^n, R>0.
    \end{equation*}
    The ``only if'' part of this equivalence is known as Frostman's Lemma, while the ``if'' part is known as the Mass Distribution Principle.

    The Hausdorff Dimension is defined as
    \begin{align*}
        \dim_H(E)= \sup &\left\{ s \in [0,n]: \mathcal{H}^s(E) = \infty \right\} \\
        = \inf &\left\{ s \in [0,n]: \mathcal{H}^s(E) = 0 \right\} \\
        =\sup &\left\{s \in [0,n]: \exists \ \mu \in \mathcal{M}(E) / \mu(B(x,R)) \lesssim R^s \ \forall \ x\in \mathbb{R}^n, R>0 \right\}.
    \end{align*}

The Hausdorff dimension of a set $E$ represents the cut point between the values of $s$ for which $\mathcal{H}^s (E)$ is strictly greater than 0 and the values of $s$ for which $\mathcal{H}^s (E)$ is 0.

The Fourier dimension of $E \subseteq \mathbb{R}^n$ is defined as
\begin{align*}
      \dim_F(S) = \sup &\left\{ s \in [0,n]: \exists \ \mu \in \mathcal{M}(S)/ |\widehat{\mu}(\xi)| \lesssim |\xi|^{-s/2} \ \forall \ \xi \in \mathbb{R}^n  \right\}.
 \end{align*}
 Broadly speaking, it conveys information about the fastest possible polynomial decay  for the Fourier transforms of all measures supported on $E$. A key result regarding both dimensions is that
 \begin{equation*}
    \dim_F(E) \le \dim_H(E)
\end{equation*}    
for any $E \subseteq \mathbb{R}^n$.

We include here, for the sake of coompleteness, the definition of dimension function needed in Theorem \ref{thm: generalized mockenhaupt}. A function $h: \mathbb{R}_{\ge 0} \rightarrow \mathbb{R}_{\ge 0}$ is said to be a dimension function if it satisfies the following three conditions:
\begin{itemize}
    \item $h(0)=0.$
    \item $h$ is increasing.
    \item $h$ is right continuous.
\end{itemize}
For a given $s \in \mathbb{R}_{\ge 0}$, the Hausdorff measure $\mathcal{H}^s$ is defined with respect to the function $h(t) = t^s$, which is clearly a dimension function. The definition of dimension function allows a reasonable definition of a measure $\mathcal{H}^h$ with all the good properties of a measure. For a detailed explanation over the generalized Hausdorff measures, the reader can consult \cite{rog70}. $h$ is called a doubling function when there exists a constant $c>0$ such that
\begin{equation*}
    h(2t) \le c h(t) \ \forall t \ge 0.
\end{equation*}

We move on now to the main construction of the set satisfying the conditions in Theorem \ref{thm: main theorem}. We denote by $\|x\| = \min_{z \in \mathbb{Z}} |x-z|$ to the distance from $x$ to its nearest integer. Given $u \in \mathbb{R}_{>0}, q \in \mathbb{N}$, define the set
\begin{equation*}
    \mathcal{N}_u \left( \frac{\mathbb{Z}}{q} \right) = \left\{ x \in \mathbb{R}/ \ \| qx \| \le qu \right\}.
\end{equation*}
Intuitively, the set $\mathcal{N}_u \left( \frac{\mathbb{Z}}{q} \right)$ consists of all the balls in $\mathbb{R}$ of radius $u$ centered on the rational numbers with denominator $q$. It is possible to define these sets in higher dimensions, but we will not need to do so in the present paper. For a more thorough discussion on how important sets can be defined in terms of the sets $\mathcal{N}_u \left( \frac{\mathbb{Z}}{q} \right)$, we refer the reader to Li and Liu's paper \cite{LL25}. Just to give a classic example, recall the sets of well approximable numbers, independently defined by Jarnik \cite{jar31} and Besicovitch \cite{bes34}, which Kaufman \cite{kau81} later proved to be the first examples of explicit Salem sets in $\mathbb{R}$ of dimension $\frac{2}{\alpha}$, $\alpha >0$; they can be written as
\begin{align*}
        E(\alpha) &= \left\{ x \in \mathbb{R}/ \|qx \| \le q^{1-\alpha} \ \textnormal{for infinitely many} \ q \in \mathbb{N} \right\} \\
        &= \bigcap_{i \in \mathbb{N}} \bigcup_{i \le q} \mathcal{N}_{q^{-\alpha}} \left( \frac{\mathbb{Z}}{q} \right).
    \end{align*}

We can now properly cite Li and Liu's Theorem.

\begin{theorem}\label{thm: li and liu}
Suppose $\beta, \gamma \ge 0$ and $2\gamma + \beta \le 1$. Then there exists an increasing sequence $(q_i)_i$ in $\mathbb{R}_{>0}$ such that the set
\begin{equation*}
    E:= \left\{ \begin{array}{lcc}  \bigcap_i \bigcup_{1 \le H \le q_i^\gamma} \mathcal{N}_{q_i^{-1}} \left( \frac{\mathbb{Z}}{Hq_i^\beta}\right), \ \textnormal{if} \ 2\gamma + \beta < 1   \\ \\ 
    \bigcap_i \bigcup_{1 \le H \le q_i^\gamma, prime} \mathcal{N}_{q_i^{-1}} \left( \frac{\mathbb{Z}}{Hq_i^\beta}\right), \ \textnormal{if} \ 2\gamma + \beta = 1 \end{array} \right.
\end{equation*}
    has Hausdorff dimension $2\gamma + \beta$ and Fourier dimension $2\gamma$. Moreover there exists a finite Borel measure $\mu$ supported on 
    \begin{equation*}
        \bigcap_i \bigcup_{q_i^\gamma /2 \le H \le q_i^\gamma, prime} \mathcal{N}_{q_i^{-1}} \left( \frac{\mathbb{Z} \smallsetminus p\mathbb{Z}}{Hq_i^\beta}\right) \cap [0,1],
    \end{equation*}
    satisfying
    \begin{equation*}
        \mu (B(x,R)) \lesssim_\epsilon R^{2\gamma + \beta - \epsilon} \ \textnormal{and} \ 
        |\widehat{\mu}(\xi)| \lesssim_\epsilon |\xi|^{-\gamma+\epsilon}, \forall \epsilon >0.
    \end{equation*}
\end{theorem}

Our intention is to modify this result in order to produce a set that satisfies the hypotheses of Theorem \ref{thm: generalized mockenhaupt}. In essence, we will need the two magnitudes that define the dimension, $\beta$ and $\gamma$, to now be functions. Let's discuss the specific functions that we will need.

We want the Fourier decay of our measure $\mu$ to be slower than any negative power function, in order to allow the Fourier dimension of our set to be zero. However, we want it to be something that we can operate with, so we want to define $\gamma(x)$ for every $x\in (0,1) \cup (1,+\infty)$ in such a way that
\begin{equation*}
    \xi^{-\gamma(\xi)} = \frac{1}{\log^r (\xi)}
\end{equation*}
for some $r>0$ and for large values of $\xi\in\mathbb{R}$. Nonetheless, we need to proceed with caution: since we are dealing with both Fourier decays and dimension functions in the same result, defining $\gamma(x)$ in this way for all $x \in \mathbb{R}_{>0}$ would result in $x^{\gamma(x)} \xrightarrow{x \rightarrow 0^+} -\infty$, which does not make sense for a dimension function. This problem does not arise when $\gamma$ is simply a constant, so we want to define $\gamma(x)$ in a way that emulates the ``symmetric'' behaviour of $x^{-\gamma}$, namely the fact that $x^{-\gamma}$ evaluated in $x^{-1}$ is exactly $x^{\gamma}$. With this in mind, we define $\gamma(x)$ on $(0,1)\cup (1,+\infty)$ as
\begin{equation}\label{eq: definition of gamma}
    \gamma(x) = \frac{r \log (|\log x|)}{|\log x|},
\end{equation}
which guarantees that
\begin{equation*}
    x^{\gamma(x)} = \left\{ \begin{array}{lcc} \log^r (x)  \ \textnormal{if} \ x > 1   \\ \\ 
    \frac{1}{\log^r(\frac{1}{x})} \ \textnormal{if} \ 0<x < 1. \end{array} \right.
\end{equation*}
It is not important to define the function in $1$, since its behaviour only matters for sufficiently large and small values of $x$.

In a similar fashion, we will define
\begin{equation}\label{eq: definition of beta}
    \beta(x) = 1-(2+a)\gamma(x),
\end{equation}
for some small $a>0$ that we will be able to choose freely. This guarantees that
\begin{equation*}
    s(x)=\beta(x)+2\gamma(x)=1-a\gamma(x),
\end{equation*}
the function that ``captures'' the Hausdorff dimension, is smaller than 1 but greater than any $0<\alpha<1$, in a way that is comparable to $\gamma$. As a consequence, we get that
\begin{equation*}
    x^{\beta(x)} = \left\{ \begin{array}{lcc} \frac{x}{\log^{(2+a)r}(x)}  \ \textnormal{if} \ x > 1   \\ \\ 
   x \log^{(2+a)r} \left( \frac{1}{x} \right) \ \textnormal{if} \ 0<x < 1, \end{array} \right.
\end{equation*}
and
\begin{equation*}
    x^{s(x)} = \left\{ \begin{array}{lcc} \frac{x}{\log^{ar}(x)}  \ \textnormal{if} \ x > 1   \\ \\ 
   x \log^{ar} \left( \frac{1}{x} \right) \ \textnormal{if} \ 0<x < 1.\end{array} \right.
\end{equation*}

After these definitions we are in a position to state our main theorem regarding the specific construction.
\begin{theorem}\label{thm: our construction theorem}
Let $r>0, a>0$. For these values, let $\gamma(x),\beta(x)$ be the functions defined in \eqref{eq: definition of gamma} and \eqref{eq: definition of beta} respectively. Then there exists an increasing sequence $(q_i)_i$ of positive real numbers such that the set
\begin{equation}\label{eq:Set E}
    E:=\bigcap_{i\in \mathbb{N}} \bigcup_{1 \le H \le q_i^{\gamma(q_i)}} \mathcal{N}_{q_i^{-1}} \left( \frac{\mathbb{Z}}{Hq_i^{\beta(q_i)}}\right)  
\end{equation}
has $\dim_H(E)=1$ and $\dim_F (E)=0$. Moreover, $E$ supports a finite Borel measure $\mu$ for which 
    \begin{equation}\label{eq: our frostman construction}
        \mu (B(x,R)) \lesssim_\epsilon R\log^{ar+\epsilon}\left( \frac{1}{R} \right) \quad  \forall R < 1/3, \forall \epsilon >0 
        \end{equation}
        and
        \begin{equation}\label{eq: our fourier construction}
        |\widehat{\mu}(\xi)| \lesssim_\epsilon \frac{1}{\log^{r-\epsilon} (|\xi|)} \quad \forall |\xi| \ge 2, \forall \epsilon >0
    \end{equation}
    hold, and the Fourier decay is optimal on $E$ (up to $\epsilon$-loss).
\end{theorem}

\begin{remark}
    Conditions \eqref{eq: our frostman construction} and \eqref{eq: our fourier construction} can be restated as
    \begin{equation*}
          \mu (B(x,R)) \lesssim_\epsilon R^{\beta(R)+2\gamma(R)} \log^\epsilon \left( \frac{1}{R} \right) \quad \textnormal{and} \quad |\widehat{\mu}(\xi)| \lesssim_\epsilon |\xi|^{-\gamma(|\xi|)}\log^\epsilon (|\xi|),
    \end{equation*}
    to be seen as an analogue of the results in Theorem \ref{thm: li and liu}.
\end{remark}

The interesting thing is that, even though the sets defined in Theorem \ref{thm: our construction theorem} are outside of the hypotheses of the original STM theorem, it is still possible to restrict the Fourier transform to these sets. That is, with the set from Theorem \ref{thm: our construction theorem} we can provide a proof of Theorem \ref{thm: main theorem}.

\begin{proof}[Proof of Theorem \ref{thm: main theorem}]
    Let $r>1, a>0$, $E$ and $\mu$ as in Theorem \ref{thm: our construction theorem}.  For every $\epsilon >0$, Theorem \ref{thm: generalized mockenhaupt} guarantees a restriction theorem $R_E (p \rightarrow 2)$ if $\left(\Gamma^{(\epsilon)}_k \right)_k \in \ell_1$, where
\begin{align*}
    {\Gamma^{(\epsilon)}_k} &= \left[ \frac{1}{\log^{r-\epsilon}(2^{k-1})}  \right]^{\frac{2}{p} -1} \left[ \log^{ar+\epsilon}(2^k)  \right]^{2-\frac{2}{p}} \\
    &\lesssim \frac{1}{k^{(r-\epsilon)(\frac{2}{p}-1)}} k^{(ar+\epsilon)(2-\frac{2}{p})}.
\end{align*}
A simple calculation shows that the right hand side belongs to $\ell_1$ if and only if
\begin{equation*}
    \frac{2}{p} (r+ar) > 1+r+ 2ar+ \epsilon. 
\end{equation*}
Since this holds for any $\epsilon>0$, a sufficient condition for a restriction theorem $R_E (p \rightarrow 2)$ is that
\begin{equation*}
    \frac{2}{p} (r+ar) > 1+ r+2ar , 
\end{equation*}
which happens if and only if
\begin{equation*}
    p < 1 + \frac{r-1}{1+r+2ar}.
\end{equation*}
\end{proof}

We discuss here some subtleties that can be observed comparing STM's theorem and our Theorem \ref{thm: main theorem}. Notice that if we apply Theorem \ref{thm:STM} for a measure satisfying condition $(i)$ for every $0<\alpha<1$ and condition $(ii)$ for some fixed 
$\beta>0$, we would obtain a restriction theorem $R_E (p \rightarrow 2)$ for $1 \le p <2$, since
\begin{equation*}
    \sup_{0<\alpha<1} \frac{2(2-2\alpha+\beta)}{4(1-\alpha)+\beta} = 2,
\end{equation*}
independently of the value of $\beta>0$. That is, with any \emph{positive} Fourier decay and full Hausdorff dimension, the full range of restriction can be obtained. Our result provides an interesting extension of this fact: with a \emph{zero} dimensional polylogarithmic decay on the Fourier transform, we can still get examples of sets with the restriction property. 
Namely, for any $\tilde{p} \in (1,2)$ it is always possible to pick $r>1, a>0$ and a corresponding set $E$ as in Theorem \ref{thm: our construction theorem} such that a restriction theorem $R_E (p \rightarrow 2)$ holds for every $1 \le p < \tilde{p}$.

\section{Proof of Theorem \ref{thm: our construction theorem}}
We include here the proof of our main construction result in Theorem \ref{thm: our construction theorem}.
\begin{proof}[Proof of Theorem \ref{thm: our construction theorem}:]
The proof is structured as follows: first we will construct the measure $\mu$ and prove it has the desired Fourier decay. Then we will prove that $\mu$ satisfies the Frostman condition; the fact that $\dim_H (E)=1$ follows automatically from this. Finally, we will prove that the Fourier decay of $\mu$ is sharp; this will imply that $\dim_F (E)=0$.

The construction we present here will provide a set that is actually contained in the set described in \eqref{eq:Set E}. Namely, consider the set $E'$ defined as:

\begin{equation*}
    E'=\bigcap_{i \in \mathbb{N}} \bigcup_{\substack{\frac{1}{2} q_i^{\gamma(q_i)} \le p \le q_i^{\gamma(q_i)}\\ p \text{ prime}}} \mathcal{N}_{q_i^{-1}} \left( \frac{\mathbb{Z} \smallsetminus p\mathbb{Z}}{pq_i^{\beta(q_i)}}\right) \cap [0,1].
\end{equation*}

Our measure $\mu$ will be supported on this subset. We will refrain from pointing out every time that a constant codified by a $\lesssim$ symbol depends on $r$ or $a$, since these are fixed throughout the proof and a constant depending on them creates no problems.

In order to prove that $E$ verifies the desired properties, we will impose two types of conditions on the sequence $(q_i)_i$: that the term $q_1$ is large enough, and that the sequence increases fast enough. For the sake of clarity, the precise set of conditions on $(q_i)_i$ will be provided after all the conditions are duly brought up.

We can always consider $q_i$ such that $q_i^{\beta(q_i)} \in \mathbb{N}$ by increasing $q_i$ if necessary. Therefore, the set $\mathcal{N}_{q_i^{-1}} \left(\frac{\mathbb{Z} \smallsetminus p\mathbb{Z}} {pq_i^\beta}\right)$ is formed by intervals of radius $q_i^{-1}$ and centers 
$\frac{m}{p q_i^{\beta(q_i)}}$, where $(m:p)=1$, and the intervals are disjoint if all the $q_i$ are large enough (see condition \ref{item: L1}). Again, if all $q_i$ are large enough (see condition \ref{item: L2}), we will have that
\begin{equation*}
    q_i^{\gamma(q_i)} = \log^r (q_i), \quad \textnormal{and} \quad q_i^{\beta(q_i)} = \frac{q_i}{\log^{(2+a)r}(q_i)}.
\end{equation*}
In addition, we will have that $\log \log q_i > 1$ for every $i \in \mathbb{N}$.

Let's construct the measure $\mu$ and prove it has the desired Fourier decay. It will be enough to prove that the Fourier decay holds over the integers by virtue of the following Lemma. While this lemma is well known for polynomial decays (see \cite{wol03}), we weren't able to find a version for general decays, so we present it here.

\begin{lemma}\label{lem: decay over integers is enough}
    Let $\mu$ be a finite Borel measure over $\mathbb{T}$ such that $|\widehat{\mu}(k)| \lesssim f(|k|)$ $\forall k \in \mathbb{Z}$, where $f: \mathbb{R}_{>0} \longrightarrow \mathbb{R}_{>0}$ is a function such that
    \begin{itemize}
        \item $f$ is decreasing
        \item $f\left( \frac{\xi}{2} \right) \lesssim f\left( \xi \right)$
        \item There exist $N \in \mathbb{N}, \xi_0 >0$ such that $f(\xi) \gtrsim_{N, \xi_0} \xi^{-N} \ \forall \xi \ge \xi_0 $. 
    \end{itemize}
    Then for any $\phi \in \mathcal{S}$, the measure $\nu$ defined by
    \begin{equation*}
        d\nu = \phi d \mu 
    \end{equation*}
    satisfies that
    \begin{equation*}
        |\widehat{\nu}(\xi)| \lesssim f(|\xi|) \quad \forall \ |\xi| \ge \xi_0 .
    \end{equation*}
\end{lemma}

\begin{proof}
    Let $|\xi| \ge \xi_0$. It is known (see \cite{wol03}) that 
    \begin{equation*}
    \widehat{\nu}(\xi) = \sum_{k \in \mathbb{Z}} \widehat{\mu}(k) \widehat{\phi}(\xi-k).
    \end{equation*}
We can therefore bound the absolute value of $\widehat{\nu}$ by
\begin{equation*}
    |\widehat{\nu}(\xi)| \le \sum_{|\xi - k|\ge |\xi|/2} |\widehat{\mu}(k)| |\widehat{\phi}(\xi-k)| + \sum_{|\xi - k| \le |\xi|/2} |\widehat{\mu}(k)| |\widehat{\phi}(\xi-k)|.
\end{equation*}
The first sum can be bounded by
\begin{align*}
    \sum_{|\xi - k|\ge |\xi|/2} |\widehat{\mu}(k)| |\widehat{\phi}(\xi-k)| &\lesssim_N \sum_{|\xi - k|\ge |\xi|/2} \| \widehat{\mu}\|_\infty \frac{1}{|\xi - k|^{N+1}} \\
    &\le 2 \mu(\mathbb{R}) \int_{|\xi|/2}^\infty \frac{1}{x^{N+1}}dx \lesssim_{N,\mu} |\xi|^{-N} \lesssim_{N, \xi_0} f(|\xi|),
\end{align*}
and the second sum can be bounded by
\begin{align*}
    \sum_{|\xi - k| \le |\xi|/2} |\widehat{\mu}(k)| |\widehat{\phi}(\xi-k)| \le f\left( \frac{|\xi|}{2} \right) \sum_{|\xi - k| \le |\xi|/2}  |\widehat{\phi}(\xi-k)|    \lesssim f(|\xi|) .
\end{align*}
This concludes with the proof of the lemma.
\end{proof}

We start with a function $\phi \in C_0^\infty (-1,1)$, $\phi \ge 0$, $\int \phi =1$. The first part of the construction is analogous to that in \cite{LL25}, so we will only make statements without many calculations. For every prime $p$, define
\begin{equation*}
    \phi_{i,p}(x) = \sum_{v \in \mathbb{Z}\smallsetminus p\mathbb{Z}} p^{-1}q_i^{1-\beta(q_i)}\phi(p^{-1}q_i^{1-\beta(q_i)}(x-v)).
\end{equation*}
$\phi_{i,p}$ is $p$-periodic, and its Fourier expansion is \begin{equation*}
    \phi_{i,p}(x)=\sum_{n \in \mathbb{Z}}\widehat{\phi}(pq_i^{\beta(q_i)-1}n)e^{2\pi i n x} -p^{-1}\sum_{m \in \mathbb{Z}}\widehat{\phi}(q_i^{\beta(q_i)-1}m)e^{2\pi i m x/p}.
\end{equation*} 
Define $\Phi_{i,p}(x) = \phi_{i,p}(p q_i^{\beta (q_i)} x)$, which is supported on $\mathcal{N}_{q_i^{-1}} \left( \frac{\mathbb{Z} \smallsetminus p \mathbb{Z}}{p q_i^{\beta(q_i)}} \right)$, and whose Fourier coefficients are
\begin{equation*}
    \widehat{\Phi_{i,p}}(k) = \left\{ \begin{array}{lcc} 
    (1-p^{-1})\widehat{\phi}(q_i^{-1}k) \quad \textnormal{if} \quad k \in pq_i^{\beta(q_i)\mathbb{Z}}  \\ 
    -p^{-1}\widehat{\phi}(q_i^{-1}k) \quad \textnormal{if} \quad k \in q_i^{\beta(q_i)}\mathbb{Z}\smallsetminus pq_i^{\beta(q_i)}\mathbb{Z}    \\ 
    0, \quad \textnormal{otherwise}.  \end{array} \right.
\end{equation*}

If we define $\mathcal{P}_i = \left\{ p \ \textnormal{prime} \in (q_i^{\gamma(q_i)}/2, q_i^{\gamma(q_i)}] \right\}$, the Prime Number Theorem guarantees that, for $q_i$ large enough (see condition \ref{item: L3}) we have that
\begin{equation*}
  \frac{1}{2} \frac{q_i^{\gamma(q_i)}}{\log (q_i^{\gamma(q_i)})} \le \# \mathcal{P}_i \le 2 \frac{q_i^{\gamma(q_i)}}{\log (q_i^{\gamma(q_i)})} = \frac{2\log^r (q_i)}{r \log \log (q_i)}.
\end{equation*}
In addition we can check that for $k \neq 0$
\begin{equation}\label{eq: divisor bound}
    \# \left\{ p \in \mathcal{P}_i / p q_i^{\beta(q_i)} | k \right\} \le
    \max \left\{ \frac{\log (|k|q_i^{-\beta(q_i)})}{\log (q_i^{\gamma(q_i)}/2)}, 0 \right\} \le \max \left\{  \frac{2\log (|k|q_i^{-\beta(q_i)})}{\log (q_i^{\gamma(q_i)})}, 0  \right\},
\end{equation}
where the last inequality holds when $q_i$ is sufficiently large (see condition \ref{item: L4})).

We analogously define 
\begin{equation*}
    F_i (x) = \left. \frac{1}{\# \mathcal{P}_i} \sum_{p \in \mathcal{P}_i} \frac{p}{p-1} \Phi_{i,p}(x) \right|_{[0,1]},
\end{equation*}
which is smooth on $\bigcup_{p \in \mathcal{P}_i} \mathcal{N}_{q_i^{-1}} \left( \frac{\mathbb{Z} \smallsetminus p \mathbb{Z}}{p q_i^{\beta(q_i)}} \right) \cap [0,1]$. It verifies that  $\widehat{F_i}(0) = 1$ and, for $k \neq 0$,
\begin{equation*}
    \widehat{F_i}(k) = \frac{1}{\# \mathcal{P}_i} \left( \# \{ p \in \mathcal{P}_i / k \in pq_i^{\beta(q_i)}\mathbb{Z}  \} - \sum_{p \in \mathcal{P}_i / k \in q_i^{\beta(q_i)}\mathbb{Z} \smallsetminus pq_i^{\beta(q_i)}\mathbb{Z}  }\frac{1}{p-1} \right) \widehat{\phi}(q_i^{-1}k).
\end{equation*}
We can use \eqref{eq: divisor bound} to bound, for $k \neq 0$,
\begin{align*}
    |\widehat{F_i}(k)| &\le 2 \frac{\log (q_i^{\gamma(q_i)})}{q_i^{\gamma(q_i)}} \left( \max \left\{  \frac{2\log \left(|k|q_i^{-\beta(q_i)}\right)}{\log (q_i^{\gamma(q_i)})}, 0  \right\}+1 \right) |\widehat{\phi}(q_i^{-1}k)| \\
    & \lesssim \left\{ \begin{array}{lcc} \frac{\log \log (q_i)}{\log^r (q_i)}|\widehat{\phi}(q_i^{-1}k)|  \ \textnormal{if} \ |k| \le q_i^{\beta(q_i)}   \\ \\ 
   \frac{\log\left(|k|q_i^{-\beta(q_i)} \right) + \log \log q_i}{\log^r (q_i)}|\widehat{\phi}(q_i^{-1}k)|. \ \textnormal{if} \ |k| \ge q_i^{\beta(q_i)}.\end{array} \right.  
\end{align*}

In order to proceed, we will need the following lemma.
\begin{lemma}\label{lem: bounded product lemma}
    Let $\psi \in C^\infty [0,1]$. Then 
    \begin{equation*}
        |\widehat{\psi F_i}(k) - \widehat{\psi}(k)| \lesssim  \| \psi \| \left\{ \begin{array}{lcc} \frac{\log \log q_i}{\log^{r} (q_i)}  \ \textnormal{if} \ |k| \le q_i   \\ \\ 
    \frac{\log \log |k|}{\log^{r}(|k|)} \ \textnormal{if} \ |k| \ge q_i. \end{array} \right.
    \end{equation*}
    where
    \begin{equation*}
        \| \psi \| = \sum_{j=-2}^{2}|\widehat{\psi}(j)|+ \| \widehat{\psi}(\cdot) \log^{t} (| \cdot |) \|_{\ell_1 (\mathbb{Z}\smallsetminus \{0\})}, \quad t= \max \{ 1,r \}
    \end{equation*}
\end{lemma}

Notice that there exists a constant $c >0$ such that
\begin{equation}\label{eq: bound for norm in lemma}
    \| \psi \| \le 5 \| \psi \|_\infty + c \| \psi'' \|_\infty \sum_{l \neq 0} \frac{\log^{t} |l|}{|l|^2},  
\end{equation}
which is a finite quantity that depends only on $\psi$ and $r$.
\begin{proof}
    For $0 \le |k| \le q_i$, we begin by noting that
    \begin{equation*}
        |\widehat{\psi F_i}(k) - \widehat{\psi}(k)| = |\sum_{l \in \mathbb{Z}} \widehat{\psi}(k-l)\widehat{F_i}(l)-\widehat{\psi}(k)| = |\sum_{l \neq 0} \widehat{\psi}(k-l)\widehat{F_i}(l)|.
    \end{equation*}
We will need to split the sum in three parts and bound each part separately. Let $A_1=\{ l \in \mathbb{Z} \smallsetminus \{ 0 \}/ |l| \ge q_i^{\beta(q_i)}, |k-l| \le \frac{|l|}{2} \}$, $B_1=\{ l \in \mathbb{Z} \smallsetminus \{ 0 \}/|l|\ge q_i^{\beta(q_i)}, |k-l| > \frac{|l|}{2} \}$ and $C = \{ l \in \mathbb{Z} \smallsetminus \{ 0 \}/|l|\le q_i^{\beta(q_i)}  \}$. Then

\begin{align*}
    \left|\sum_{l \in A_1} \widehat{\psi}(k-l) \widehat{F_i}(l)   \right| &\lesssim  \sum_{l \in A_1} |\widehat{\psi}(k-l)| \frac{\log(|l|q_i^{-\beta(q_i)}) + \log \log q_i}{\log^r (q_i)}|\widehat{\phi}(q_i^{-1}l)| \\
    &\le \sum_{l \in A_1} |\widehat{\psi}(k-l)| \frac{\log\left( 2q_i^{1-\beta(q_i)} \right) + \log \log q_i}{\log^r (q_i)} \\
    &\lesssim \| \psi \| \frac{\log \log q_i}{\log^{r}(q_i)},
\end{align*}
where in the second inequality we have used that $\|\widehat{\phi} \|_\infty = 1$ and $l \in A_1 \implies \frac{|l|}{2} \le |k| \le \frac{3}{2} |l| \implies |l| \le 2|k| \le 2q_i$, and in the third inequality holds since condition \ref{item: L4} implies that $2 \le \log^{r/2}q_i$. On the other hand,
\begin{align*}
    \left|\sum_{l \in B_1} \widehat{\psi}(k-l) \widehat{F_i}(l)   \right| &\lesssim  \sum_{l \in B_1} |\widehat{\psi}(k-l)| \frac{\log(|l|q_i^{-\beta(q_i)}) + \log \log q_i}{\log^r (q_i)}|\widehat{\phi}(q_i^{-1}l)| \\
    &\le  \sum_{l \in B_1} |\widehat{\psi}(k-l)|\frac{\log(2|k-l|)}{\log^r (q_i)}  \lesssim \| \psi \| \frac{1}{\log^{r}(q_i)},
\end{align*}
where in the second inequality we have used that $l \in B_1 \implies |l|\le 2|k-l|$ and that $q_i$ is large enough (see condition \ref{item: L5}). Finally,
\begin{align*}
    \left|\sum_{l \in C} \widehat{\psi}(k-l) \widehat{F_i}(l) \right| 
    &\lesssim \sum_{l \in C} |\widehat{\psi}(k-l)| \frac{\log \log (q_i)}{\log^r (q_i)}|\widehat{\phi}(q_i^{-1}l)| \le \| \psi \|\frac{\log \log (q_i)}{\log^r (q_i)}.
\end{align*}

Now, if $q_i \le |k|$ we will need to split the sum using $A_2=\{ l \in \mathbb{Z} \smallsetminus \{ 0 \}/ |l|\ge q_i^{\beta(q_i)}, |k-l| \le \frac{|k|}{2} \}$, $B_2=\{ l \in \mathbb{Z} \smallsetminus \{ 0 \}/ |l|\ge q_i^{\beta(q_i)}, |k-l| > \frac{|k|}{2} \}$ and $C = \{ l \in \mathbb{Z} \smallsetminus \{ 0 \}/|l|\le q_i^{\beta(q_i)}  \}$. For the first case, we will need the following technical result, which is easy to verify.

\begin{remark}\label{rem: technical remark}
    For $\alpha>e$, 
$
f(x) = x\log \left( e\alpha \frac{\log^{(2+a)r}(x)}{x} \right)
$    is increasing on $[e,\alpha]$.
\end{remark}

Then we have that
\begin{align*}
    \sum_{l \in A_2} |\widehat{\psi}(k-l)||\widehat{F_i}(l)| &\lesssim \sum_{l \in A_2} |\widehat{\psi}(k-l)| \frac{\log (|l|q_i^{-\beta(q_i)}) + \log \log q_i}{\log^r (q_i)} \frac{q_i}{|l|} \frac{q_i}{|l|} \\ 
    &\le \sum_{l \in A_2} |\widehat{\psi}(k-l)| \frac{\log (|l| q_i^{-\beta(q_i)}) + \log \log q_i}{\log^r (|k|)} \frac{|k|}{|l|} \frac{q_i}{|l|} \\
    &\lesssim \sum_{l \in A_2} |\widehat{\psi}(k-l)| \frac{q_i \log (e |k| q_i^{-\beta(q_i)}) + q_i \log \log q_i }{ |l| \log^r (|k|)}  \\
    &\le \sum_{l \in A_2} |\widehat{\psi}(k-l)| \frac{|k|\log (e \log^{(2+a)r}(|k|)) + |k| \log \log |k|}{|l| \log^r (|k|)} \\
    &\lesssim \| \psi \| \frac{\log \log (|k|)}{\log^r (|k|)} 
\end{align*}
where in the first inequality we have used that there is a constant $c>0$ such that $|\widehat{\phi}(\xi)| \le c |\xi|^{-2}$, the second inequality holds if $q_i$ is sufficiently large (see condition \ref{item: L6}), in the third inequality we have used that $l \in A \implies \frac{|k|}{2} \le |l| \le \frac{3}{2} |k|$, in the fourth inequality we have invoked the previous Remark, and the last inequality holds if $q_i$ is large enough (see condition \ref{item: L7}).

For $l \in B_2$, we can use that $|\widehat{F_i}(l)|\le 1$ to obtain that
\begin{align*}
    \sum_{l \in B_2
    }|\widehat{\psi}(k-l)| |\widehat{F_i}(l)| &\le \sum_{l \in B_2} |\widehat{\psi}(k-l)| \frac{\log^{r}(|k|)}{\log^{r}(|k|)} \\
    &\le \sum_{l \in B_2} |\widehat{\psi}(k-l)| |\log^{r} (2|k-l|)| \frac{1}{\log^{r}(|k|)} \\
    &\lesssim \| \psi \| \frac{1}{\log^{r} (|k|)}.
\end{align*}
Finally,
\begin{equation*}
      \left|\sum_{l \in C} \widehat{\psi}(k-l) \widehat{F_i}(l)  \right | 
    \lesssim |\sum_{l \in C} \widehat{\psi}(k-l) \frac{\log \log (q_i)}{\log^r (q_i)}|\widehat{\phi}(q_i^{-1}l)| \le \| \psi \|\frac{\log \log (|k|)}{\log^r (|k|)},
\end{equation*}
where the last inequality holds if $q_i$ is large enough (see condition \ref{item: L7}).

\end{proof}

After proving this, we can define the sequence of functions
$G_m$ as
\begin{equation}\label{eq: definition of Gm}
    G_0 = \chi_{[0,1]}, \quad G_m = \prod_{i=0}^m F_i.
\end{equation}
We will need to bound $\| G \|$. Recalling \eqref{eq: bound for norm in lemma} and noticing that
\begin{equation*}
    G''_m = \sum_{j=1}^m F''_j \prod_{{i = 1}\atop{i \neq j}}^{m} F_i + \sum_{j_1, j_2 = 1 \atop j_1 \neq j_2}^m F'_{j_1} F'_{j_2} \prod_{i=1 \atop i \neq j_1, j_2}^m F_i ,
\end{equation*}
this amounts to bound $\| F \|_\infty, \|F'\|_\infty $ and $\| F'' \|_\infty $. A straightforward calculation shows that 
\begin{align*}
    \| F_i \|_\infty &\lesssim_r  \log^{ar}(q_i) \log \log (q_i)  \| \phi \|_\infty , \\
    \| F'_i \|_\infty &\lesssim_r q_i \log^{ar}(q_i) \log \log (q_i)  \| \phi' \|_\infty , \\
    \| F''_i \|_\infty &\lesssim_r q^2_i\log^{ar}(q_i) \log \log (q_i)  \| \phi'' \|_\infty ,
\end{align*}
and as such we get that
\begin{align}
    \| G_m \| &\le  5\| G_m \|_\infty +  c_r\| G''_m \|_\infty \nonumber \\
    &\le  c_r \left[ \prod_{i=1}^m \| F_i \|_\infty + \sum_{j=1}^m \| F''_j \|_\infty \prod_{{i = 1}\atop{i \neq j}}^{m} \| F_i \|_\infty + \sum_{j_1, j_2 = 1 \atop j_1 \neq j_2}^m \| F'_{j_1} \|_\infty \| F'_{j_2} \|_\infty \prod_{i=1 \atop i \neq j_1, j_2}^m \| F_i \|_\infty \right] \nonumber \\
    &\le c_{r,\phi, m} \left[ 1+ \sum_{1 = j_1, j_2}^m q_{j_1} q_{j_2} \right]  \prod_{i=1}^m \log^{ar} (q_i) \log \log (q_i). \label{eq: bound for the norm of Gm}
\end{align}
This bound depends only on $r, a, \phi, m$ and the first $m$ terms of the sequence. As such, the sequence can be constructed to grow fast enough (see condition \ref{item: F1}) so that it guarantees
\begin{equation*}
    \| G_m \| \le \log \log (q_{m+1}).
\end{equation*}

Along with Lemma \ref{lem: bounded product lemma}, this yields that
\begin{equation}\label{eq: almost perfect bound on the difference}
    | \widehat{G_{m+1}}(k) - \widehat{G_m}(k)  | \lesssim  \left\{ \begin{array}{lcc} \frac{\log^2 \log (q_{m+1})}{\log^{r} (q_{m+1})}  \ \textnormal{if} \ |k| \le q_{m+1}   \\ \\ 
     \frac{\log \log (q_{m+1}) \log \log |k|}{\log^{r}(|k|)} \ \textnormal{if} \ |k| \ge q_{m+1}. \end{array} \right.
\end{equation}
Notice that the constant in \eqref{eq: almost perfect bound on the difference} does not depend on $m$, since the constant in \eqref{eq: bound for the norm of Gm} was absorbed by the choice of $q_{m+1}$.

Define $\mu_m = G_m \mathcal{L}$ (a multiple of the Lebesgue measure). Noting that $\widehat{G_1}(0) = \widehat{F_1}(0)$ = 1, we get that
\begin{equation}\label{eq: condition for positive mass}
    |\widehat{G_{m+1}}(0)-1| \le \sum_{i=1}^m |\widehat{G_{i+1}}(0)-\widehat{G_i}(0)| 
    \le c \sum_{i=1}^\infty \frac{\log^2 \log (q_{i+1})}{\log^{r}(q_{i+1})}.    
\end{equation}
If $(q_i)_i$ increases fast enough (see condition \ref{item: F2}) we can guarantee that the RHS in \eqref{eq: condition for positive mass} is less than 1/2. This implies that for every $m \in \mathbb{N}$, $\mu_m (\mathbb{R}) \le 3/2$, and therefore there exists a weak limit $\mu$ for (a subsequence of) $(\mu_m)_m$ (see, for example, Corollary 21.19 in \cite{schi17}). Taking the limit on $m$ in \eqref{eq: condition for positive mass}, we can guarantee that $\widehat{\mu}(0) >0$ and $\mu$ is not the zero measure.  It is clear that 
\begin{equation*}
    \textnormal{supp} (\mu) \subseteq \limsup_{m} \textnormal{supp} (G_m) = E'.
\end{equation*}

We now proceed to analyze the decay of $\widehat{\mu}$. Let $k \notin \{ -1,0,1 \}$ and $m$ such that $q_{m+1} \ge |k| $. We then have that
\begin{align*}
    |\widehat{G_{m+1} }(k)| &= |\widehat{G_{m+1} }(k) - \widehat{G_0}(k) | \\
    &\le \sum_{i=0}^m |\widehat{G_{i+1} }(k) - \widehat{G_i}(k)  | \\
    &= \sum_{i/ q_{i+1} \le |k| } |\widehat{G_{i+1} }(k) - \widehat{G_i}(k)  |
    + \sum_{i / q_{i+1}\ge |k| } |\widehat{G_{i+1} }(k) - \widehat{G_i}(k)  | \\
    &\lesssim \sum_{i/ q_{i+1} \le |k| } \frac{\log \log q_{i+1} \log \log |k|}{\log^{r} (|k|)}
    + \sum_{i/ q_{i+1} \ge |k| } \frac{\log^2 \log (q_{i+1} )}{\log^{r} (q_{i+1} )} \\
    &\le   \frac{1}{\log^{r-\epsilon}(|k|)} \left( \sum_{i/ q_{i+1} \le |k| } \frac{\log^2 \log |k|}{\log^{\epsilon} |k|}   +\sum_{i/ q_{i+1} \ge |k| } \frac{\log^2 \log q_{i+1}}{\log^{\epsilon}(q_{i+1} )}  \right) \\
    & \lesssim_\epsilon \frac{1}{\log^{r-\epsilon}(|k|)} \left( \sum_{i/ q_{i+1} \le |k| } \frac{1}{\log^{\epsilon/2} |k|}   +\sum_{i/ q_{i+1} \ge |k| } \frac{1}{\log^{\epsilon/2}(q_{i+1} )}  \right)\\
    &\lesssim_\epsilon \frac{1}{\log^{r-\epsilon}(|k|)}  \sum_{i \in \mathbb{N}} \frac{1}{\log^{\epsilon/2}(q_{i+1})},
\end{align*}
where in the fourth inequality we have used that there exists some constant $c_\epsilon$ such that $\log \log x \le c_\epsilon \log^{\epsilon/4}(x)$ for every $x>e^e$. 

A fast enough increase of $(q_i)_i$ (see condition \ref{item: F3}) guarantees that 
\begin{equation*}
    |\widehat{G_{m+1} }(k)| \lesssim_\epsilon \frac{1}{\log^{r-\epsilon}(|k|)}.
\end{equation*}
Taking limit on $m$ gives us the desired decay for $\widehat{\mu}$ over the integers.

We now want to prove the Frostman condition. Fix $R<1/3$. Our first task is noticing that this can be reduced to proving
\begin{equation}\label{eq: final reduction to frostman on steps}
    \mu_{m_0}(B(x,2R)) \lesssim_{\epsilon} 2R\log^{ar+\epsilon}\left( \frac{1}{2R} \right)
\end{equation}
for a suitable $m_0$, depending on $R$, that will be specified later. 

Indeed, suppose that \eqref{eq: final reduction to frostman on steps} holds for some $m_0$. Let $\psi$ be a fixed $C^\infty$ function such that 
\begin{equation*}
    \psi (x) \left\{ \begin{array}{lcc}  = 1 \quad \textnormal{if} \quad |x| \le 1 \\
     \in (0,1) \quad \textnormal{if} \quad 1 < |x| < 2 \\
     = 0 \quad \textnormal{otherwise}
    \end{array} \right.
\end{equation*}
and notice that, for any measure $\nu$, we have that
\begin{equation*}
    \int \psi \left( 2\frac{x-y}{R} \right) d\nu(y) \le \nu(B(x,R)) \le \int \psi \left( \frac{x-y}{R} \right) d\nu(y).
\end{equation*}
Applying this to $\mu_{m_0}$ we get that
\begin{align}
    \mu_{m_0}(B(x,2R)) &\ge \int \psi \left( \frac{x-y}{R} \right) G_{m_0}(y) dy \nonumber \\
    &= \int R \widehat{\psi}(R \xi) e^{2\pi i \xi x} \widehat{G_{m_0}}(\xi)d\xi. \label{eq: first transformed integral}
\end{align}
Similarly, we get that 
\begin{equation}
    \mu(B(x,R)) \le \int R \widehat{\psi}(R \xi) e^{2\pi i \xi x} \widehat{\mu}(\xi)d\xi \label{eq: second transformed integral}.
\end{equation}
Notice that, since $q_i$ was chosen large enough (see condition \ref{item: L7}), \eqref{eq: almost perfect bound on the difference} implies that
\begin{equation*}
    | \widehat{G_{m+1}}(k) - \widehat{G_m}(k)  | \lesssim   \frac{\log^2 \log (q_{m+1})}{\log^{r} (q_{m+1})} \quad \forall k \in \mathbb{Z}.
\end{equation*}
Since $G_{m+1}-G_m$ is supported on a compact subset of $(0,1)$, multiplying by a suitable function in $C^\infty([0,1])$ and invoking Lemma \ref{lem: decay over integers is enough} we get that
\begin{equation*}
\| \widehat{G_{m+1}} - \widehat{G_m}  \|_\infty \lesssim   \frac{\log^2 \log (q_{m+1})}{\log^{r} (q_{m+1})}.
\end{equation*}

We can use this along \eqref{eq: first transformed integral} and \eqref{eq: second transformed integral} to bound
\begin{align}
     \mu(B(x,R)) - \mu_{m_0}(B(x,2R))  &\le \int R |\widehat{\psi}(R \xi)| |\widehat{\mu}(\xi) - \widehat{G_{m_0}}(\xi) |d\xi \nonumber \\
    &\le \int R |\widehat{\psi}(R \xi)| \sum_{m=m_0}^\infty|\widehat{G_{m+1}}(\xi) - \widehat{G_{m}}(\xi) |d\xi \nonumber\\
    &\le c \sum_{m=m_0}^\infty \frac{\log^2 \log q_{m+1}}{\log^{r}(q_{m+1})} \| \widehat{\psi} \|_1 \le q_{m_0}^{-1} \label{eq: mu and mu0 are close}
\end{align}
where in the second inequality we have used that $\mu$ is the weak limit of $\mu_m$, and the fourth inequality holds if $(q_i)_i$ increases fast enough (see condition \ref{item: F4}). If $m_0$ is large enough so that $q_{m_0}^{-1} \le R $, \eqref{eq: mu and mu0 are close} implies that
\begin{align*}
    \mu(B(x,R)) &\lesssim_{\epsilon} R + 2R \log^{ar+\epsilon}\left( \frac{1}{2R} \right) \\
    &\le \left( \frac{1}{\log^{ar+\epsilon}\left( \frac{3}{2} \right)} +2 \right)R \log^{ar+\epsilon}\left( \frac{1}{2R} \right) \\
    &\lesssim_{\epsilon} R \log^{ar+\epsilon}\left( \frac{1}{R} \right),
\end{align*}
where in the second inequality we have used that $2R < 2/3$.

 Having done this, we now proceed proceed to prove that \eqref{eq: final reduction to frostman on steps} holds for some $m_0$ such that $q_{m_0} \ge R^{-1}$. We first need to bound $|G_{m}(y)|$. Recalling the definitions of $\Phi_{i,p}, F_i$ and $G_{m_0}$, let us prove that $\Phi_{i,p}$ and $\Phi_{i,\Tilde{p}}$ have disjoint supports for $p \neq \tilde{p}$. For this, it suffices to prove that
\begin{equation*}
    B\left( \frac{v}{p q_i^{\beta(q_i)} }, q_i^{-1} \right) \cap B\left( \frac{w}{\Tilde{p}q_i^{\beta(q_i)}}, q_i^{-1}  \right) = \varnothing \quad \forall v \in \mathbb{Z} \smallsetminus p\mathbb{Z}, w \in \mathbb{Z} \smallsetminus \Tilde{p}\mathbb{Z}.
\end{equation*}
This is easily checked by noticing that the distance between the centers of any two such balls is
\begin{align}
    \left| \frac{v}{pq_i^{\beta(q_i)} } -\frac{w}{\Tilde{p}q_i^{\beta(q_i)} } \right| &= \frac{1}{p\Tilde{p}q_i^{\beta(q_i)} } |\tilde{p}v-pw| \ge \frac{1}{p\Tilde{p}q_i^{\beta(q_i)} } \label{eq: distance between centers}\\ &\ge q_i^{-2\gamma(q_i)-\beta(q_i) } = \frac{\log^{ar}(q_i)}{q_i} = q_i^{-s(q_i)} > 2q_i^{-1}, \nonumber
\end{align}
where the first inequality is due to $p \nmid v$ and $\tilde{p} \nmid w$, and the last inequality holds for every $i \in \mathbb{N}$ when $q_1$ is chosen to be large enough (see condition \ref{item: L8}).

Once that we have proven that the supports of $\Phi_{i,p}$ and $\Phi_{i,\tilde{p}} $ are disjoint, we can use this to bound $|F_i (y)|$:
\begin{align*}
    |F_i (y)| &\le \frac{1}{\# \mathcal{P}_i} \sum_{p \in \mathcal{P}_i} \frac{p}{p-1}|\Phi_{i,p}(y)| \le \frac{2}{\# \mathcal{P}_i} \max_{p \in \mathcal{P}_i} \| \Phi_{i,p} \|_\infty \chi_{\bigcup_{p \in \mathcal{P}_i}  \textnormal{Supp}(\Phi_{i,p})} (y) \\
    &\le \frac{2 \|\phi\|_\infty q_i^{1-\gamma(q_i)-\beta(q_i) } }{\# \mathcal{P}_i} \chi_{\bigcup_{p \in \mathcal{P}_i} \mathcal{N}_{q_i^{-1}} \left( \frac{\mathbb{Z} \smallsetminus p\mathbb{Z} }{p q_i^{\beta(q_i)}} \right) }(y).
\end{align*}
Finally, this guarantees that
\begin{equation*}
    |G_m (y)| \le \prod_{i=1}^m |F_i (y)| \le C_{\phi, m} \prod_{i=1}^m \frac{q_i^{1-\gamma(q_i)-\beta(q_i)} }{\# \mathcal{P}_i } \chi_{\bigcap_{i=1}^m  \bigcup_{p \in \mathcal{P}_i} \mathcal{N}_{q_i^{-1}} \left( \frac{\mathbb{Z} \smallsetminus p\mathbb{Z} }{p q_i^{\beta(q_i)}} \right) }(y).
\end{equation*}
We can use this to bound
\begin{align}
    &\int_{B(x,2R)} G_{m}(y)dy \nonumber \\ &\quad \le C_{\phi, m} \prod_{i=1}^m \frac{q_i^{1-\gamma(q_i)-\beta(q_i)} }{\# \mathcal{P}_i } q_m^{-1} \# \{ q_m^{-1}-\textnormal{intervals that intersect} \ B(x,2R) \}; \label{eq: useful bound for the integral of Gm}
\end{align}
the reduction to intervals of radius $q_m^{-1}$ amounts to the inclusion
\begin{equation*}
    \bigcap_{i=1}^m  \bigcup_{p \in \mathcal{P}_i} \mathcal{N}_{q_i^{-1}} \left( \frac{\mathbb{Z} \smallsetminus p\mathbb{Z} }{p q_i^{\beta(q_i)}} \right) \subseteq \bigcup_{p \in \mathcal{P}_m} \mathcal{N}_{q_m^{-1}} \left( \frac{\mathbb{Z} \smallsetminus p\mathbb{Z} }{p q_m^{\beta(q_m)}} \right).
\end{equation*}

Now in order to get the desired bound, we work with $m=m_0$ such that
\begin{equation*}
    q_{m_0}^{-s(q_{m_0})} \le R < q_{m_0 -1}^{-s(q_{m_0 -1})}.
\end{equation*}
Since this $m_0$ clearly satisfies that $q_{m_0}^{-1} < R $, it is a valid index for deducing the Frostman condition from \eqref{eq: final reduction to frostman on steps}. First, notice that since the centers of the $q_{m_0}-$intervals are precisely $$\bigcup_{p \in \mathcal{P}_{m_0} } \frac{\mathbb{Z}\smallsetminus p\mathbb{Z}}{pq_{m_0}^{\beta(q_{m_0})}},$$ the amount of such intervals included in $(x-2R,x+2R)$ is at most
\begin{equation*}
    8R q_{m_0}^{\beta(q_{m_0}) + \gamma(q_{m_0})} \# \mathcal{P}_{m_0},
\end{equation*}
and we can combine this with \eqref{eq: useful bound for the integral of Gm} to get that
\begin{align}
    \int_{B(x,2R)} G_{m_0}(y)dy &\le C_{\phi,m_0} \left( \prod_{i=1}^{m_0} \frac{q_i^{1-\gamma(q_i)-\beta(q_i)}}{\# \mathcal{P}_i} \right) R q_{m_0}^{\beta(q_{m_0})+\gamma(q_{m_0})-1} \# \mathcal{P}_{m_0} \nonumber\\
    &= C_{\phi,m_0} \left( \prod_{i=1}^{m_0-1} \frac{q_i^{1-\gamma(q_i)-\beta(q_i)}}{\# \mathcal{P}_i}\right) R . \label{eq: first of the two bounds for combination}
\end{align}
Now, for a given $\epsilon>0$, fix $\tilde{\epsilon}= \frac{\epsilon}{ar}$, and rewrite \eqref{eq: first of the two bounds for combination} as
\begin{equation*}
    \underbrace{C_{\phi, m_0} \left( \prod_{i=1}^{m_0-1} \frac{q_i^{1-\gamma(q_i)-\beta(q_i)}}{\# \mathcal{P}_i}\right) R^{(1+\tilde{\epsilon})a\gamma(R)}}_{(I)} \underbrace{R^{1-(1+\tilde{\epsilon})a\gamma(R)}}_{(II)}.
\end{equation*}
Since $R<1/3$ implies that
\begin{equation*}
    (II) = R\log^{ar+\epsilon}\left( \frac{1}{R} \right) \le 
    \frac{1}{2} \left[ 1+ \frac{\log 2}{\log \left( \frac{3}{2} \right)} \right]^{ar+\epsilon} 2R\log^{ar+\epsilon}\left( \frac{1}{2R} \right),
\end{equation*}
we only need to bound $(I)$ by a constant depending on $\epsilon$. In order to do this, we will first analyse only the factors of $(I)$ which depend solely on $q_{m_0 -1}$ and $R$. But since $x^{\gamma(x)}$ is increasing for $0<x<1$ and $R< q_{m_0 -1}^{-s(q_{m_0 -1})}$, we get that these factors can be controled by 
\begin{align*}
    &\frac{q_{m_0 -1}^{1-\gamma(q_{m_0 -1}) -\beta(q_{m_0 -1})  }}{\# \mathcal{P}_{m_0 -1}} R^{(1+\tilde{\epsilon})a\gamma(R)} \\
    &\quad \le 
    2 q_{m_0 -1}^{a\gamma(q_{m_0 -1})} \log \left(q_{m_0 -1}^{\gamma (q_{m_0 -1})} \right) \left[ \left(q_{m_0 -1}^{-s(q_{m_0 -1})} \right)^{\gamma \left(  q_{m_0 -1}^{-s(q_{m_0 -1})} \right)}   \right]^{(1+\tilde{\epsilon})a} \\
    &\quad = 2\log^{ar}(q_{m_0 -1}) r \log \log (q_{m_0 -1}) \frac{1}{\log^{(1+\tilde{\epsilon})ar} \left( q_{m_0 -1}^{s(q_{m_0 -1})} \right)} \\
    &\quad = \frac{2\log^{ar}(q_{m_0 -1}) r \log \log (q_{m_0 -1})}{\left( \log q_{m_0 -1} -ar \log \log q_{m_0 -1} \right)^{ar+\epsilon}} 
    \le \frac{2^{1+ar+\epsilon} r \log \log (q_{m_0 -1})}{ \log^{\epsilon} q_{m_0 -1} },
\end{align*}
where the last inequality holds if $q_1$ is large enough (see condition \ref{item: L9}). 
If $(q_i)_i$ increases fast enough (see condition \ref{item: F5}) we can guarantee that
\begin{equation}\label{eq: bounded by a constant on epsilon}
    (I) \le C_{\phi, m_0} \left( \prod_{i=1}^{m_0-2} \frac{q_i^{1-\gamma(q_i)-\beta(q_i)}}{\# \mathcal{P}_i}\right) \frac{2^{1+ar+\epsilon} r \log \log (q_{m_0 -1})}{ \log^{\epsilon} q_{m_0 -1} } \lesssim C_\epsilon ,
\end{equation}
as desired. 

This proves that, under all the conditions that we have imposed on the sequence $(q_i)_i$, the measure $\mu$ that we have constructed satisfies the Frostman condition and has the desired Fourier decay over the integers. A simple application of Lemma \ref{lem: decay over integers is enough} yields that the Frostman condition and desired Fourier decay hold as stated in the theorem (possibly for a measure different than $\mu$, but supported on $E$ nonetheless).  It is now possible to construct an explicit sequence $(q_i)_i$ which in turn can be used to construct the set $E$, by constructing any sequence which verifies all of the conditions that can be seen explicitly stated in the appendix. We would like to point out that the sequence $(q_i)_i$ does not depend on $\epsilon$.

Finally, we will prove the sharpness of the Fourier decay. Let $r'>r$ and $\gamma'(x) = \frac{r'}{r}\gamma(x) = r' \frac{\log | \log (x) |}{|\log (x) |}$. Suppose there exists some probability measure $\mu$ supported on $E$ such that $\widehat{\mu}(\xi) \lesssim |\xi|^{-\gamma' (|\xi|)} = \frac{1}{\log^{r'}(|\xi|)}$ for $|\xi| \ge 2$. Then there exists a subsequence $(q_{i_j})_j$ of $(q_i)_i$ and some $\nu$ supported in
\begin{equation*}
    \Tilde{E} = \bigcap_{j = 1}^\infty \bigcup_{1 \le H \le q_{i_j}^{\gamma(q_{i_j})}} \mathcal{N}_{2 q_{i_j}^{-1+\beta(q_{i_j})}} \left( \frac{\mathbb{Z}}{H} \right)
\end{equation*}
such that 
\begin{equation*}
    |\nu (\xi)| \lesssim_\epsilon |\xi|^{-\frac{r'}{r(2+a)}+\epsilon} \quad \forall \epsilon>0.
\end{equation*}
This constitutes a contradiction, since
\begin{equation*}
    q_{i_j}^{-1+\beta(q_{i_j})} = q_{i_j}^{-(2+a)\gamma(q_{i_j})} = \left( q_{i_j}^{\gamma(q_{i_j})} \right)^{-(2+a)}
\end{equation*}
implies that $\dim_H (\Tilde{E}) = \frac{2}{2+a}$ via standard arguments. This proof will be fairly analogous to the proof of $\mu$'s Fourier decay. Thus, we will take the liberty of omitting some details.

Multiplying by a suitable smooth function and applying Lemma \ref{lem: decay over integers is enough} if necessary, we can assume $\textnormal{supp}(\mu)\subseteq (0,1)$ so that there is a $\delta >0$ such that $\textnormal{supp}(\mu)\subseteq [\delta,1-\delta]$. Consider only $q_i$ such that $q_i^{-1}<\delta$. Take $\phi \in C^{\infty}[-1,1]$, nonnegative, $\int \phi =1$. Define
\begin{equation*}
    \phi_i (x) = q_i \phi (q_i x), \quad \textnormal{supp}(\phi_i) \subseteq [-q_i^{-1}, q_i^{-1}],
\end{equation*}
so that supp$(\mu \ast \phi_i) \subseteq [\delta - q_i^{-1}, 1- \delta + q_i^{-1}] \subseteq [0,1]$. Define
\begin{equation*}
    F_i (x) = \left. \sum_{v \in \mathbb{Z}} q_i^{- \beta(q_i)} \mu \ast \phi_i \left( q_i^{-\beta(q_i)}(x-v) \right) \right|_{[0,1]    }
\end{equation*}
and note that indeed 
$$supp(F_i) \subseteq \bigcup_{1 \le H \le q_{i_j}^{\gamma(q_{i_j})}} \mathcal{N}_{2 q_{i_j}^{-1+\beta(q_{i_j})}} \left( \frac{\mathbb{Z}}{H} \right). $$

Since $F_i$ is (the restriction to $[0,1] $ of) the periodization of $q_i^{-\beta(q_i)} \mu \ast \phi_i (q_i^{-\beta(q_i)}\cdot ) $ we can bound its Fourier coefficients for $k \neq 0$ as
\begin{equation}\label{eq: Fourier control for F_i the second time}
    |\widehat{F_i}(k)| = |\widehat{\mu \ast \phi_i}(q_i^{\beta(q_i)}k)  | \lesssim  \left( q_i^{\beta(q_i) }|k| \right)^{- \gamma' \left( q_i^{\beta(q_i) } |k|\right)} |\widehat{\phi}(q_i^{-1+\beta(q_i)}k)|
\end{equation}
if $q_i$ is large enough so that $q_i^{\beta(q_i)} \ge 2$; while
\begin{equation*}
    \widehat{F_i}(0) = 1.
\end{equation*}

We will now need to prove an analogue of Lemma \ref{lem: bounded product lemma}
\begin{lemma}\label{lem: another bounded product lemma}
    Suppose $\psi \in C^\infty([0,1])$. Then
\begin{equation*}
    |\widehat{\psi F_i}(k)- \widehat{\psi}(k)| \lesssim \| \psi \| 
    \left\{ \begin{array}{lcc}  \frac{1}{\log^{r'}(q_i)}  \ \textnormal{if} \ |k| \le q_i^{1-\beta(q_i)}   \\ \\ 
    |k|^{-\frac{r'}{r(2+a)}} \ \textnormal{if} \ |k| \ge q_i^{1-\beta(q_i)}. \end{array} \right.
\end{equation*}
where $\| \psi \| = |\widehat{\psi}(0)|+\sum_{l \in \mathbb{Z}\smallsetminus 0} | \widehat{\psi}(l)| |l|^{\frac{r'}{r(2+a)}}$.
\end{lemma}

\begin{proof}
    As in Lemma \ref{lem: bounded product lemma}, the proof comes down to bounding
    \begin{equation*}
        \left| \sum_{l \neq 0} \widehat{\psi}(k-l)\widehat{F_i}(l)  \right| \le
        \sum_{l \in A_2} |\widehat{\psi}(k-l)||\widehat{F_i}(l)| + \sum_{l \in B_2} |\widehat{\psi}(k-l)||\widehat{F_i}(l)|,
    \end{equation*}
where
\begin{equation*}
    A_2 = \left\{ l \neq 0 / |k-l|>\frac{|k|}{2} \right\}, \quad B_2 = \left\{ l \neq 0 / |k-l| \le \frac{|k|}{2} \right\}.
\end{equation*}

For the first sum, we invoke \eqref{eq: Fourier control for F_i the second time} to get
\begin{align*}
    \sum_{l \in A_2} |\widehat{\psi}(k-l)||\widehat{F_i}(l)| &\lesssim   \sum_{l \in A_2} |\widehat{\psi}(k-l)| \left( q_i^{\beta(q_i) }|l| \right)^{- \gamma' \left( q_i^{\beta(q_i) } |l|\right)} |\widehat{\phi}(q_i^{-1+\beta(q_i)}l)| \\
    &= \sum_{l \in A_2} |\widehat{\psi}(k-l)| \frac{|\widehat{\phi}(q_i^{-1+\beta(q_i)}l)|}{\log^{r'}(q_i^{\beta(q_i)} |l|)} \\
    &\le \frac{1}{(\log (q_i)-(2+a)r \log \log (q_i))^{r'}} \sum_{l \in A_2} |\widehat{\psi}(k-l)| \frac{|k-l|^{\frac{r'}{r(2+a)}}}{|k-l|^{\frac{r'}{r(2+a)}}} \\
    &\lesssim \| \psi \| \frac{1}{\log^{r'}(q_i)} \sup_{l \in A_2} |k-l|^{-\frac{r'}{r(2+a)}} .
\end{align*}
where in the second inequality we have used that $|l|\ge 1$ and $\| \widehat{\phi} \|_\infty \le 1$, and the third inequality holds for sufficiently large $q_i$.

The fact that $k \notin A_2$ implies that $|k-l|\ge 1$ for $l \in A_2$, and therefore immediately yields the bound we need for $0 \le |k| \le q_i^{1-\beta (q_i)}$. For $|k| \ge q_i^{1-\beta(q_i)}$, the desired bound follows after imposing $q_i \ge e$ and noting that for $l \in A_2$,
\begin{equation*}
    |k-l|^{-\frac{r'}{r(2+a)}} \le \left( \frac{|k|}{2} \right)^{-\frac{r'}{r(2+a)}} \lesssim  |k|^{-\frac{r'}{r(2+a)}}
\end{equation*}

Now consider the sum over $l \in B_2$. For $0 \le |k| \le q_i^{1-\beta(q_i)}$, we again use that $|l| \ge 1$ and $\| \widehat{\phi} \|_\infty \le 1$ to get that
\begin{align*}
    \sum_{l \in B} |\widehat{\psi}(k-l)||\widehat{F_i}(l)| &\lesssim \sum_{l \in B} |\widehat{\psi}(k-l)| \frac{1}{\log^{r'}(q_i^{\beta(q_i)}|l|)} \\
    &\lesssim \| \psi \| \frac{1}{\log^{r'}(q_i)},
\end{align*}
where the second inequality holds for sufficiently large $q_i$.

For $l \in B_2$ and $|k| \ge q_i^{1-\beta(q_i)}$, we use  that 
\begin{equation}\label{eq: bound the transform by a negative power}
    \left|\widehat{\phi}\left(q_i^{-1+\beta(q_i)}|l| \right) \right| \lesssim (q_i^{-1+\beta(q_i)}|l|)^{-\frac{r'}{r(2+a)}}
\end{equation}
to get that
\begin{align*}
    \sum_{l \in B_2} |\widehat{\psi}(k-l)||\widehat{F_i}(l)| &\lesssim  \sum_{l \in B_2} |\widehat{\psi}(k-l)| (q_i^{\beta(q_i)}|l|)^{-\gamma'(q_i^{\beta(q_i)}|l|)} (q_i^{-1+\beta(q_i)}|l|)^{-\frac{r'}{r(2+a)}} \\
    &= \sum_{l \in B_2} |\widehat{\psi}(k-l)| \frac{|l|^{-\frac{r'}{r(2+a)}} \left( \log^{(2+a)r} (q_i) \right)^{\frac{r'}{r(2+a)}}}{(\log |l| + \log (q_i) - (2+a)r \log \log (q_i))^{r'}} \\
    &\lesssim \| \psi \| |k|^{-\frac{r'}{r(2+a)}},
\end{align*}
where  in the second inequality we have used that $|l| \ge 1$ and that $|l| \ge |k|/2$ for $l \in B_2$.
\end{proof}

Consider a subsequence $(q_{i_j})_j$ of $(q_i)_i$ such that $q_{i_1} $ satisfies all the conditions we have needed on $q_i$ thus far. We proceed analogously and define
\begin{equation}
    G_0 = \chi_{[0,1]}, \quad G_m = \prod_{j=0}^m F_{i_j}.
\end{equation}
We will need to bound $\| G \|$, but this calculation is entirely analogous so we get
\begin{equation*}
    \| G_m \| \le c_{m,r,r',a} \prod_{j=1}^m q_{i_j}^{c_{r,r',a}}
\end{equation*}
and since this bound depends only on the parameters and the first $m$ terms of the subsequence, it allows us to pick $q_{i_{m+1}} $ such that
\begin{equation*}
    \| G_m \| \le \log \log (q_{i_{m+1}} ),
\end{equation*}
which in turn guarantees that
\begin{equation}\label{eq: final step bound on G_m}
    |\widehat{G_{m+1}}(k)- \widehat{G_m}(k)| \lesssim  
    \left\{ \begin{array}{lcc}  \frac{\log \log q_{i_{m+1}}}{\log^{r'}(q_i)}  \ \textnormal{if} \ 0 \le |k| \le q_{i_{m+1}}^{1-\beta(q_{i_{m+1}})}   \\ \\ 
    \log \log (q_{i_{m+1}}) |k|^{-\frac{r'}{r(2+a)}} \ \textnormal{if} \ |k| \ge q_{i_{m+1}}^{1-\beta(q_{i_{m+1}})}. \end{array} \right.
\end{equation}

We analogously define $\nu$ as the weak limit of $G_m \lambda$. For $k=0$, we get that
\begin{equation*}
    | \widehat{G_{m+1}}(0) -1 | \lesssim \sum_{j=1}^\infty \frac{\log \log (q_{i_{m+1}})}{\log^{r'} (q_{i_{m+1}})}.
\end{equation*}
If $q_{i_j} $ increases fast enough (analogously to condition \ref{item: F2}) this yields
\begin{equation*}
    |\widehat{G_{m+1}}(0)-1| \le \frac{1}{2}.
\end{equation*}
Taking limit on $m$ guarantees that the measure $\nu$ is not the zero measure. To prove the desired Fourier decay, pick $k \neq 0$ and $m$ such that $q_{i_{m+1}}^{1-\beta(q_{i_{m+1}})} \ge |k|$ so that
\begin{align*}
    |\widehat{G_{m+1}}(k)| &\le \sum_{j=0}^m |\widehat{G_{j+1}}(k) - \widehat{G_j}(k) | \\
    &\le  \sum_{j \in A_3} |\widehat{G_{j+1}}(k) - \widehat{G_j}(k) | + \sum_{j \in B_3} |\widehat{G_{j+1}}(k) - \widehat{G_j}(k) |,
\end{align*}
where
\begin{equation*}
    A_3 = \left\{ j/ q_{i_{j+1}}^{1-\beta(q_{i_{j+1}})} \le |k| \right\}, \quad B_3 = \left\{ j/ q_{i_{j+1}}^{1-\beta(q_{i_{j+1}})} \ge |k| \right\}.
\end{equation*}

For the first sum we can use \eqref{eq: final step bound on G_m} to get
\begin{align*}
    \sum_{j\in A_3} |\widehat{G_{j+1}}(k) - \widehat{G_j}(k) | &\lesssim \sum_{j\in A_3} \frac{ \log \log (q_{i_{j+1}})}{|k|^{\frac{r'}{r(2+a)}}}  \\ &= |k|^{-\frac{r'}{r(2+a)}+\epsilon} \sum_{j\in A_3} \frac{\log \log (q_{i_{j+1}})}{|k|^\epsilon}\\
    &\le |k|^{-\frac{r'}{r(2+a)}+\epsilon} \sum_{j\in A_3} \frac{\log \log (q_{i_{j+1}})}{\log^{\epsilon r (2+a)}(q_{i_{j+1}})},
\end{align*}
while for the second sum \eqref{eq: final step bound on G_m} yields
\begin{align*}
    \sum_{j\in B_1} |\widehat{G_{j+1}}(k) - \widehat{G_j}(k) | &\lesssim  \sum_{j\in B_1} \frac{\log \log (q_{i_{j+1}})}{\log^{r'} (q_{i_{j+1}})}\\
    &= |k|^{-\frac{r'}{r(2+a)}+\epsilon} \sum_{j\in B_1} \frac{\log \log (q_{i_{j+1}}) |k|^{\frac{r'}{r(2+a)}-\epsilon}}{\log^{r'} (q_{i_{j+1}})} \\
    &\le |k|^{-\frac{r'}{r(2+a)}+\epsilon} \sum_{j\in B_1} \frac{\log \log (q_{i_{j+1}}) \left( \log^{(2+a)r} q_{i_{j+1}} \right)^{\frac{r'}{(2+a)r}-\epsilon}}{\log^{r'} (q_{i_{j+1}})} \\
    &\le |k|^{-\frac{r'}{r(2+a)}+\epsilon} \sum_{j\in B_1} \frac{\log \log (q_{i_{j+1}})}{\log^{\epsilon r(2+a)} q_{i_{j+1}}}.
\end{align*}
All in all, we get that
\begin{equation*}
    |\widehat{G_{m+1}}(k)| \lesssim |k|^{-\frac{r'}{r(2+a)} + \epsilon} \sum_{j \ge 0} \frac{\log \log (q_{i_{j+1}})}{\log^{\epsilon r(2+a)}(q_{i_{j+1}})}.
\end{equation*}
We can choose $(q_{i_j})_j$ to grow fast enough (analogously to condition \ref{item: F3}) so that the series always converges, with the growth rate independent of $\epsilon$. Finally
, taking limit on $m$ amounts to
\begin{equation*}
    |\widehat{\nu}(k)| \lesssim_{\epsilon, \{ q_{i_j} \} } |k|^{-\frac{r'}{r(2+a)}+\epsilon}
\end{equation*}
as desired.

\end{proof}

\section{Appendix}

We summarize in this appendix all the conditions that are needed to be imposed on the sequence $\{q_i\}$ in order to produce the results in Theorem \ref{thm: our construction theorem}. As a first elementary condition we impose that $q_i^{\beta(q_i)}  \in \mathbb{N}$.

Then, the full list of conditions on the sequence $(q_i)_i$ is the following, separated into two subclasses: the ``L'' conditions describing that the terms of the sequence must be \textbf{L}arge enough, and the ``F'' conditions describing how \textbf{F}ast the sequence should increase.

The \textbf{L} conditions: $q_1$ should be large enough so that 

\begin{enumerate}[(L\arabic*)]
    \item $q_1 > e^{2^{\frac{1}{(1+a)r}}}$. For intervals with radius $q_i^{-1}$ and centers $\frac{m}{pq_i^{\beta(q_i)}}$, $(p:m)=1$, this guarantees that the minimal distance between any two centers, $\frac{1}{pq_i^{\beta(q_i)}}$ is larger that $2q_i^{-1}$, and as such these intervals are disjoint. \label{item: L1} 

    \item $q_1 > e^e$. Since $e^e >1$, this fixes the formula for $\gamma(q_i)$ and $\beta(q_i)$. Moreover, this guarantees that for every $i \in \mathbb{N}$, $\log \log q_i$ is a real number greater than 1.
    \label{item: L2}
    
    \item $q_1 > e^{\lambda^{1/r}}$, where $\lambda$ is the number such that $$x \ge \lambda \implies \frac{1}{2} \le \frac{\# \mathcal{P} \cap [x/2,x]}{\frac{x}{\log x}} \le 2$$ (Prime Number Theorem). \label{item: L3}
    
    \item  $q_1 > e^{4^{1/r}}$. This implies $\frac{1}{2} r \log \log (q_i) \ge \log 2$, and as such $$\log \left( q_i^{\gamma(q_i)}/{2} \right) \ge \frac{1}{2} \log \left(q_i^{\gamma(q_i)}\right) \quad \textnormal{for every} \ i.$$ \label{item: L4}

    \item  $q_1 > \lambda_{a,r}$, where $\lambda_{a,r}$ is the number such that
    \begin{equation*}
        x > \lambda_{a,r} \implies \frac{x}{\log x} \ge (2+a)r+1.
    \end{equation*}
    This guarantees that for every $i$
    \begin{equation*}
        -\log (q_{i}) + [(2+a)r+1] \log \log (q_{i}) \le 0.
    \end{equation*}\label{item: L5}

    \item $q_1 > e^r$. This guarantees that all the $q_i$ are in the domain where the function $\frac{x}{\log^r (x)}$ is increasing.
    \label{item: L6}

    \item $q_1 > \exp(\exp(1/r))$. This guarantees two things for every $i \in \mathbb{N}$. First, that $e \le \log^r q_i$; second, that the function $\frac{\log \log x}{\log^r x}$ is decreasing on $x \ge q_i$. \label{item: L7}

    \item $q_1 > \exp (2^{\frac{1}{ar}})$. This guarantees that for every $i \in \mathbb{N}$ $q_i^{-s(q_i)} > 2q_i^{-1}$.
    \label{item: L8}

    \item $q_1 > \lambda_{a,r}$, where $\lambda_{a,r}$ is the number such that
    \begin{equation*}
        x > \lambda_{a,r} \implies \log x \le x^{\frac{1}{2ar}}.
    \end{equation*}
    This guarantees that for every $i$
    \begin{equation*}
        \log (q_{i+1}) - ar\log \log (q_{i+1}) \ge \frac{1}{2} \log (q_{i+1}).
    \end{equation*}
    \label{item: L9}

    \item $q_1 > e^{\lambda_r}$, where $\lambda_r$ is the number such that
    \begin{equation*}
        x > \lambda_r \implies \log x \le x^{\frac{r}{4}}.
    \end{equation*}
    This guarantees that for every $i$
    \begin{equation*}
        \log \log (q_{i+1}) \le \log^{\frac{r}{4}}(q_{i+1}).
    \end{equation*}
    \label{item: L10}

\end{enumerate}

The \textbf{F} conditions: $(q_i)_i$ should increase fast enough so that
\begin{enumerate}[(F\arabic*)]

    \item $q_{m+1} \ge \exp ( \exp (2cm^2 q_m^2 \log^{m(ar+1)}(q_m)))$, with $c=c_{r,\phi,m}$ the constant in \eqref{eq: bound for the norm of Gm}. This guarantees that
    \begin{align*}
        \| G_m \| &\le c \left[ 1+ \sum_{1 = j_1, j_2}^m q_{j_1} q_{j_2} \right]  \prod_{i=1}^m \log^{ar} (q_i) \log \log (q_i) \\ &\le 2cm^2 q_m^2 \log^{m(ar+1)}(q_m) \le \log \log (q_{m+1}).
    \end{align*}
    \label{item: F1}
    
    \item $q_i \ge \exp \left( (c 2^{i} )^{{2}/{r}} \right)$, with $c$ the constant in \eqref{eq: condition for positive mass}. This, along with condition \ref{item: L10}, guarantees that
    \begin{equation*}
    |\widehat{G_{m+1}}(0)-1| \le c \sum_{i=1}^\infty \frac{\log^2 \log (q_{i+1})}{\log^{r}(q_{i+1})} \le c \sum_{i=1}^\infty \frac{2^{-(i+1)}}{c}  = \frac{1}{2}.
\end{equation*}
    \label{item: F2}
    
    \item $q_i \ge e^{2^i}$. This guarantees that the sum
    \begin{equation*}
        \sum_{i \in \mathbb{N}} \frac{1}{\log^{\epsilon/2}(q_{i+1})}
    \end{equation*}
    converges to a finite value which can be bounded by a constant that depends only on $\epsilon$.
    \label{item: F3}
    
    \item $q_{i+1} \ge \max \left\{ 2 q_i, \exp \left( \left( 2c \| \widehat{\psi} \|_1 q_i \right)^{\frac{2}{r}} \right) \right\}$, with $c$ the constant in \eqref{eq: mu and mu0 are close}. This, along with condition \ref{item: L10}, guarantees that 
    \begin{align*}
        c\| \widehat{\psi} \|_1 \sum_{m = m_0}^\infty \frac{\log^2 \log q_{m+1}}{\log^{r} q_{m+1}} &\le c\| \widehat{\psi} \|_1 \sum_{m = m_0}^\infty \frac{1}{\log^{\frac{r}{2}} q_{m+1}} \\
        &\le \sum_{m=m_0}^{\infty} \frac{1}{2 q_m} \le \sum_{m=m_0}^\infty \frac{1}{2^{m+1-m_0}q_{m_0}} \le q_{m_0}^{-1}.
    \end{align*}
    \label{item: F4}
    
    \item $q_{i+1} \ge \exp \exp \left( c Q_i \right)$ where $c = C_{\phi, i+2}$ the constant in  \eqref{eq: bounded by a constant on epsilon} and
    \begin{equation*}
        Q_i = \prod_{j=1}^{i} \frac{q_j^{1-\gamma(q_j)-\beta(q_j)}}{\# \mathcal{P}_j}.
    \end{equation*}
    This guarantees that 
    \begin{align*}
        c Q_{m_0 -2} \frac{2^{1+ar+\epsilon} r \log \log (q_{m_0 -1})}{\log^\epsilon q_{m_0 -1}} &\le 2^{1+ar+\epsilon} r c_\epsilon \frac{c Q_{m_0 -2}}{\log^{\epsilon/2} q_{m_0 -1}}\\ 
        &\le 2^{1+ar+\epsilon} r c_\epsilon \frac{c Q_{m_0 -2}}{e^{\frac{\epsilon}{2} c Q_{m_0 -2}}} \\
        &\le 2^{1+ar+\epsilon} r c_\epsilon d_\epsilon,
    \end{align*}
    where $\displaystyle c_\epsilon = \max_{x \in [e^e, \infty)} \frac{\log \log x}{\log^{\epsilon/2}x}$ and $\displaystyle d_\epsilon = \max_{x \in [0,\infty]} \frac{x}{e^{\frac{\epsilon}{2}x}}$.
    \label{item: F5}

\end{enumerate}

\section{Acknowledgements}
This research was supported by grants: PICT 2018-3399
(ANPCyT), PICT 2019-03968 (ANPCyT) and CONICET PIP 11220210100087.

\providecommand{\bysame}{\leavevmode\hbox to3em{\hrulefill}\thinspace}
\providecommand{\MR}{\relax\ifhmode\unskip\space\fi MR }
\providecommand{\MRhref}[2]{%
  \href{http://www.ams.org/mathscinet-getitem?mr=#1}{#2}
}
\providecommand{\href}[2]{#2}

\end{document}